\documentclass[a4paper,11pt,reqno]{amsart}
\usepackage{amsmath,a4wide}
\usepackage{xfrac}
\usepackage{stmaryrd,mathrsfs,bm,amsthm,mathtools,yfonts,amssymb,color}
\usepackage{esint}
\usepackage{xcolor}
\usepackage[pagebackref,citecolor=black,linkcolor=blue,colorlinks=true]{hyperref}
\usepackage{tikz,tikz-3dplot}

\begingroup
\newtheorem{theorem}{Theorem}[section]
\newtheorem{lemma}[theorem]{Lemma}
\newtheorem{proposition}[theorem]{Proposition}
\newtheorem{corollary}[theorem]{Corollary}
\endgroup

\theoremstyle{definition}
\newtheorem{definition}[theorem]{Definition}
\newtheorem{remark}[theorem]{Remark}

\numberwithin{equation}{section}

\newcommand{\N}{\mathbb{N}} 
\newcommand{\R}{\mathbb{R}} 
\newcommand{\C}{\mathbb{C}} 

\newcommand{\A}{\mathcal{A}} 
\newcommand{\G}{\mathcal{G}} 

\newcommand{\sep}{{\rm sep}}

\newcommand\res{\mathop{\hbox{\vrule height 7pt width .3pt depth 0pt\vrule height .3pt width 5pt depth 0pt}}\nolimits}

\newcommand{\M}{\mathbf{M}} 

\newcommand{\reg}{\mathrm{Reg}} 
\newcommand{\sing}{\mathrm{Sing}} 

\newcommand{\eps}{\varepsilon} 
\newcommand{\spt}{\mathrm{spt}} 
\newcommand{\supp}{\mathrm{supp}} 
\newcommand{\dist}{\mathrm{dist}} 
\newcommand{\diam}{\mathrm{diam}} 
\renewcommand{\epsilon}{\varepsilon}

\def\XXint#1#2#3{{\setbox0=\hbox{$#1{#2#3}{\int}$ }
		\vcenter{\hbox{$#2#3$ }}\kern-.6\wd0}}

\DeclareMathOperator{\Div}{div} 

\def\I#1{{\mathcal{A}}_{#1}}

\newcommand{\Iqs}{{\mathcal{A}}_Q(\R^{n})}
\newcommand{\Iq}{{\mathcal{A}}_Q}
\def\a#1{\left\llbracket{#1}\right\rrbracket}
\newcommand{\norm}[1]{\left\lVert#1\right\rVert} 

\newcommand{\Iqsn}{\overset{\circ}{\mathcal{A}_Q}(\R^{n})}
\newcommand{\Iqn}{\overset{\circ}{\mathcal{A}_Q}}

\newcommand\cH{{\mathcal{H}}}

\newcommand\cG{{\mathcal{G}}}

\title[$2$-d stationary $Q$-valued maps]{Interior regularity for two-dimensional stationary $Q$-valued maps}

\author[J. Hirsch]{Jonas Hirsch}
\address{Mathematisches Institut, Universit\"at Leipzig, Augustusplatz 10, D-04109 Leipzig, Germany}
\email{hirsch@math.uni-leipzig.de}
\author[L. Spolaor]{Luca Spolaor}
\address{Department of Mathematics, UC San Diego, AP\&M, La Jolla, California, 92093, USA}
\email{lspolaor@ucsd.edu}

\begin{document}

\maketitle

\begin{abstract}We prove that $2$-dimensional $Q$-valued maps that are stationary with respect to outer and inner variations of the Dirichlet energy are H\"older continuous and that the dimension of their singular set is at most one. In the course of the proof we establish a strong concentration-compactness theorem for equicontinuous maps that are stationary with respect to outer variations only, and which holds in every dimensions. 
   
\end{abstract}

\tableofcontents

\section{Introduction}

Recall that a $Q$-valued map $f\in W^{1,2}(\Omega,\Iq(\R^n))$, $\Omega\subset \R^m$ open, is \emph{stationary (with respect to outer and inner variations of the Dirichlet energy)} if it is a critical point of the Dirichlet energy, that is it satisfies an \emph{outer variation formula}
\begin{equation}\label{eq:outer}
    \mathcal O(f, \psi):=\int \sum_i\langle Df_i(x)\,:\, D_x \psi(x,f_i(x))\rangle \,dx+\int \sum_i\langle Df_i(x)\,:\, D_u \psi(x,f_i(x))\cdot Df_i(x)\rangle \,dx=0\,,
\end{equation}
where $\psi(x,u) \in C^\infty(\Omega\times \R^n;\R^n)$ with $\Omega$ compact and 
\[
|D_u\psi|\leq C<\infty\qquad \text{and}\qquad |\psi|+|D_x\psi|\leq C\, (1+|u|)\,,
\]
and an \emph{inner variation formula}
\begin{equation}\label{eq:inner}
    \mathcal I(f, \phi):=2 \int \sum_{i=1}^Q \langle Df_i \,:\, Df_i\cdot D\phi \rangle - \int |Df|^2 \, \Div \phi = 0\,,\qquad \forall \phi \in C^\infty_c(\Omega, \R^m)\,.
\end{equation}
We will say that a $Q$-valued map $f\in W^{1,2}(\Omega,\Iq(\R^n))$, $\Omega\subset \R^m$ open, is \emph{weakly stationary} if $\mathcal O(f, \cdot)=0$, that is if it is stationary with respect to outer variations only.

For the derivation of such formulas as the Euler-Lagrange equations of the appropriate Dirichlet energy and the related Sobolev theory of $Q$-valued maps, we refer the reader to \cite[Sections 2 and 3]{DS}, whose notations we will follow.

Multivalued maps that minimize an appropriate Dirichlet energy were introduced by Almgren in \cite{Alm} in his celebrated proof of the optimal bound on the dimension of the singular set of area minimizing currents in high codimension, as the appropriate linearized problem. More recently, De Lellis and Spadaro revisited this theory with metric techniques in \cite{DS1,DLS_Currents,DLS_Center,DLS_Blowup,DS} (see also \cite{Hi} for minimizers taking values in a smooth compact Riemannian manifold, and \cite{BDPW} for minimizers of the $p$-Dirichlet energy). 

In order to extend Almgren's proof and techniques to the study of the singular set of stationary varifolds and currents, it is therefore a natural first step to investigate the regularity of stationary multivalued maps: this is the goal of this note in the $2$-dimensional case. The only known result in this setting is due to Lin who proved in \cite{Lin} continuity of such maps. In this paper we give a shorter and intrinsic proof of Lin's result (that is without the use of Almgren's embedding) and we improve it to H\"older regularity. We then prove a concentration compactness result for maps that are stationary with respect to outer variations only, and equicontinuous, and which holds in every dimension. Finally we combine these results to prove the optimal bound on the dimension of the singular set of $2$-dimensional stationary $Q$-valued maps.

\subsection{Main results}
Our first result is a continuity result for $2$-dimensional stationary maps. As mentioned above, this has already been proven in \cite{Lin}, following ideas from \cite{Grueter1}. Our proof also follows the blueprint of Gr\"uter, but with respect to Lin's proof it has the advantage of avoiding Almgren's projection, thus being completely intrinsic, and moreover we obtain an explicit modulus of continuity that will be crucial in the proofs of the results below. We refer the reader to Section \ref{ss:notation} for the relevant notations.

\begin{theorem}[Continuity]\label{thm:cont}
 Let $f\in W^{1,2}(B_1,\Iq(\R^n))$ be stationary in $B_{1}\subset \R^2$. Then $f$ is continuous in $B_{1/2}$, more precisely there exists a constant $C=C(Q)>0$ such that
 \begin{equation}\label{eq:modulusofcont}
     \G(f(x),f(y))^2\leq C \left(\int_{B_{3|x-y|}(x)}|Df|^2 \right) + C\,\left(\int_{B_1} |Df|^2\right) |x-y|^2 \,,
 \end{equation}
 for every $x,y\in B_{1/2}$ with $|x-y|<\frac16$.
\end{theorem}
We remark that for the higher dimensional case the best known result is VMO regularity proved by Bouafia, De Pauw and Wang in \cite{BDPW}. Our result only holds in dimension $2$ since we make use of conformal methods in the proof of Theorem \ref{thm:cont}. If such a theorem was available in higher dimensions, then all the results below would extend with the same proofs.

Next we prove higher integrability for $2$-dimensional stationary maps. A similar result for Dir-minimizing multivalued maps in any dimension has been proven with different techniques in \cite[Theorem 5.1]{DS1}.  

\begin{theorem}[Higher integrability]\label{thm:highint}
There exist $p>2$ and a constant $C=C(Q)> 0$ such that if  $f\in W^{1,2}(B_1,\Iq(\R^n))$ is stationary in $B_{1}\subset \R^2$ then 
\begin{equation}
    \left( \int_{B_{\frac12}} |Df|^p \right)^\frac1p \le C \int_{B_1} |Df|^2\,. 
\end{equation}
\end{theorem}

\begin{remark}\label{rem.higherIntegrability}
In \cite{Spadaro} it had been proven that in 2 dimensional Dirichlet minimizers the optimal higher integrability is $p<\frac{2Q}{Q-1}$. Our method does not provide the sharp integrability.  
\end{remark}

As an immmediate corollary of Theorems \ref{thm:cont} and \ref{thm:highint}, we deduce H\"older regularity of $2$-dimensional multivalued stationary maps. It is interesting to notice that our proof makes no use of Almgren's embedding, at difference from the proof in the minimizing case. 

\begin{corollary}[H\"older continuity]\label{cor:hoelder}
There exist positive constants $C$ and $\alpha\in (0,1)$, depending only on $Q$, such that if $f\in W^{1,2}(B_1, \Iq(\R^n))$ is a stationary map in $B_1\subset \R^2$, then 
\begin{equation}\label{eq:holder}
    \G(f(x),f(y)) \leq C \left(\int_{B_1} |Df|^2\,dx\right)^\frac12\,|x-y|^\alpha \,,\qquad \forall x,y\in B_{1/2}\,.
\end{equation}
\end{corollary}

 We remark that Theorem \ref{thm:cont} and Corollary \ref{cor:hoelder} fail for weakly stationary maps, that is maps that are stationary for outer variations only, as shown by the following example.

\begin{remark}\label{rem.outerVariationInsufficient}
	For any $0<\epsilon<1$ we consider the function
	\[ u_\epsilon(x) = \left(1- \frac{\ln|x|}{\ln\epsilon}\right)_+\,,\quad x\in \R^2.\]
Note that $u_\epsilon$ is harmonic outside of $B_\epsilon$ and the sequence $(u_\eps)_\eps$ converges pointwise to $1$ for every $|x|>0$ as $\epsilon \downarrow 0$. 
Furthermore we can compute its energy to be
	\[ \int_{B_R} |Du_\epsilon|^2  = \frac{2\pi}{(\ln\epsilon)^2} (\ln R-\ln\epsilon)\,,\]
which converges to $0$ as $\epsilon \to 0$.
	Next define the $2$-valued map
	\[f_\epsilon(x)= \a{u_\epsilon(x)} + \a{-u_\epsilon(x)}\,.\]
	It is stationary with respect to the outer variation since 
	\[\mathcal{O}(f_\epsilon,\psi)= \int D u_\epsilon \cdot D(\psi(x,u_\epsilon)) - D u_\epsilon \cdot D(\psi(x,-u_\epsilon))= - \int_{\partial B_\epsilon} \frac{\partial u_\epsilon}{\partial \nu} \psi(x,0) - \frac{\partial u_\epsilon}{\partial \nu} \psi(x,0) = 0\,. \]
	Moreover we have for any $\sigma < 1 $ and $|x|\ge \epsilon^\sigma$
	\[\frac{\G(f_\epsilon(x), f_\epsilon(0))^2}{\int_{B_1}|Df_\epsilon|^2 }= \frac{|\ln\epsilon|}{2\pi} \left(1- \frac{\ln|x|}{\ln\epsilon}\right)_+\ge \frac{|\ln\epsilon|}{2\pi}(1-\sigma) \to + \infty \text{ as } \epsilon \downarrow 0\,.  \]
It follows that Corollary \ref{cor:hoelder} fails for outer stationary maps that are not stationary for inner variations. On the other hand, in Section \ref{subsec.higherintegrability} we will show that any $Q$-valued map with modulus of continuity of the form \eqref{eq:modulusofcont} satisfy Corollary \ref{cor:hoelder}, and therefore Theorem \ref{thm:cont} fails too for weakly stationary maps.
\end{remark} 

In fact it is possible to modify the above example showing the stronger result that weakly stationary maps are not in general continuous.

\begin{proposition}\label{prop:weaknotcont}
    There exists a weakly stationary map $f\in W^{1,2}(B_1,\I2(\R))$ which is not continuous at $0$.
\end{proposition}

Next we prove a concentration compactness result for weakly stationary equicontinuous maps valid in every dimension. Given the decomposition in Subsection \ref{sec:splitting}, we will not need the full power of the following statement in the rest of the paper (i.e., we will only need the compactness part). However we state it in the most general form with possible future applications in mind.  
It is worth noticing that it holds in every dimension.

\begin{theorem}[Concentration compactness for equicontinuous and outer stationary maps]\label{thm:comp}
Let $(f_j)_j \subset W^{1,2}(B_1, \Iq(\R^n))$, $B_1\subset \R^m$, be a sequence of  maps such that for every $j\in \N$ the following properties hold:
\begin{enumerate}
    \item there exists a modulus continuity $\omega$, independent of $j$, such that
    \[
    \G(f_j(x),f_j(y))\leq \omega(|x-y|)\,,\qquad \forall x,y\in B_1\,, 
    \]
    \item $\int_{B_1}|Df_j|^2\,dx\leq C<\infty$;
    \item $\mathcal O(f_j, B_1)=0$.
\end{enumerate}
Then there exist $q\in \N$, vectors $a_1(k),\dotsc a_q(k)\in \R^n$ and multivalued functions $h_1\dotsc,h_q\in W^{1,2}(B_1,\I{Q_i}(\R^n))$, continuous with modulus of continuity $\omega$ and such that 
$$
\sum_{i=1}^q Q_i=Q
\quad\text{and}\quad \mathcal O(h_i, B_1)=0\,,\,\;\text{for all $i=1,\dotsc,q$} \,,$$ 
and  satisfying
\begin{gather}
\lim_{k\to \infty}\sup_{B_r} \G(g_k(x),f_k(x))=0 \label{eq:conccomp1}\\
\lim_{k\to \infty}\int_{B_r}\left||Df_k|^2-|Dg_k|^2 \right|=0  \,, \label{eq:conccomp2}
\end{gather}
where $g_k(x):=\sum_{i=1}^qa_i(k)\oplus h_i$. 
\end{theorem}

Finally we will use the previous theorems to prove the optimal bound on the dimension of the singular set of $2$-dimensional stationary $Q$-valued maps. We recall that a point \emph{$x\in \Omega$ is regular} if there exists a neighborhood $B\subset \Omega$ of $x$ and $Q$ analytic functions $f_i \colon B \to \R^n$ such that 
\[
f(y) = \sum_{i=1}^Q \a{f_i(y)}\qquad \text{ for
almost every }y \in B\,,
\]
and either $f_i(x)\neq f_j(x)$ for every $x\in B$ or $f_i \equiv f_j$. The \emph{singular set $\Sigma_f$ of $f$} is the complement in $\Omega$ of the set of regular points.

\begin{theorem}[Dimension of the singular set]\label{thm:dimension} Let $f\in W^{1,2}(\Omega, \Iq(\R^n))$, with $\Omega \subset \R^2$, be a stationary map. Then $\dim_{\cH}(\Sigma_f)\leq 1$.
\end{theorem}

It should be possible to improve the above result to $1$-rectifiability of the singular set and locally finite length, by using the same techniques as in \cite{DMS} (see Remark \ref{rm:mink}). We also point out that under additional regularity on the map $f$ more structure on the set of branch points is known thanks to the work of Krummel and Wikramasekera, \cite{KW1,KW2}.

\subsection{Differences with Dir-minimizing case} The proofs of the above theorems in the Dir-minimizing case are achieved via comparison with suitable competitors and using induction on the multiplicity $Q$. Of course, no competitor argument is available in our setting. Moreover, while if a Dir-minimizing map splits as the sum of two maps, then such maps are Dir-minimizing separately, this is not the case for stationary maps, since the inner variation doesn't split, as the following example illustrates. Notice that, thanks to Corollary \ref{cor:hoelder} together with Lemma \ref{lem.OuterInvariance} and \ref{lem.InnerInvariance}, the following example requires $Q\ge 4$.

\begin{remark}
Consider the the 2-valued map
\[g(x)=\begin{cases} \a{x} + \a{-x} &\text{ for } x> 0 \\ 2\a{0} & \text{ for } x< 0 \end{cases}\,.\]
One can directly check that $g$ satisfies the outer variation since $g$ is the composition of linear hence harmonic functions on $x>0$ and $x<0$ furthermore it satisfies $\lim_{x\uparrow 0} g(x)= 2\a{0}=\lim_{x\downarrow0 } g(x)$. Nonetheless it does not satisfy the inner variation since 
\[|g'(x)|^2=\begin{cases} 2 &\text{ for } x>0 \\ 0 &\text{ for } x <0 \end{cases}\,.\]
If instead we consider for any value $a \in \R$ the 4-valued map
\[ f(x) = (g(x)\oplus a)+ (g(-x) \oplus (-a))\,,\]
then we obtain a stationary map, since both $g(x)$ and $g(-x)$ are stationary for the outer variation and $f$ satisfies trivially the inner variations since $|f'(x)|^2=2$ for all $x$. Choosing $a$ sufficient large we obtain therefore a stationary map that splits locally in two pieces which are not stationary separately.
\end{remark}

Finally, another obvious difference is the dimension of the singular set: in minimizing case classical singularities do not appear and the codimension of the singular set is $2$, in the stationary case it is at least $1$. This makes the classification of $2$-dimensional homogeneous global stationary maps harder, as illustrated in the following example.

\begin{remark}\label{rem:examplehomogeneous}
	Let us fix two points $T_{\pm} \in \Iqn(\R^n)$ such that $|T_+|=|T_-|\neq0$. Consider the $1$-homogeneous stationary map 
	\[ h(x) = \begin{cases}
			x T_+ &\text{ if } x>0\\
			x T_- &\text{ if } x\le 0
		\end{cases}\,.\]
		and its extension to $\C$ by $h(x,y)=h(x)$. 
		Finally we may consider for every $Q_1>1$ the map 
		\[f(z)= \sum_{w^{Q_1} = z} h(w)\,.\]
		It is an exercise to check that $f$ is stationary itself. 
\end{remark}

\subsection{Acknowledgements} 
The first author was partially supported by the German Science Foundation DFG in context of the Priority Program SPP 2026 ``Geometry at Infinity''. The second author acknowledges the support of the NSF Career Grant DMS 2044954. 
\subsection{Conflict of interests and data availability statements:}
The authors have no competing interests to declare. Data sharing is not applicable to this article as no data sets were generated or analysed during the current study.

\section[Preliminaries and splitting]{Preliminaries and splitting $Q$-points}
In this section we recall some basic notations and definitions about stationary $Q$-valued maps, and we introduce a useful combinatoric lemma.

\subsection{Preliminaries and notations}\label{ss:notation}
For the sake of completeness we recall some of the  basic notations and definitions concerning $Q$-valued maps. For a detailed treatment we refer the reader to \cite{DS}, whose notation and conventions we follow.

We denote by $\a{p_i}$ the Dirac mass in $p_i\in \R^n$ and we define the \emph{space of $Q$-points}
\[
\Iq(\R^n):=\left\{\sum_{i=1}^Q\a{p_i}\,:\,p_i\in \R^n \text{ for every }i=1\dotsc,Q \right\}\,.
\]
We define a \emph{distance function on $\Iq(\R^n)$} by setting for every $T_1,T_2\in \Iq(\R^n)$, with $T_1=\sum_{i=1}^Q\a{p_i}$ and $T_2=\sum_{i=1}^Q\a{s_i}$, 
\[
\mathcal G(T_1,T_2):=\min_{\sigma \in \mathcal P_Q} \sqrt{\sum_i\left|p_i-s_{\sigma(i)} \right|^2}\,,
\]
where $\mathcal P_Q$ denotes the group of permutations of $\{1,\dotsc,Q\}$. If $T=\sum_{i=1}^Q\a{p_i}$, then we define the \emph{support}, the \emph{diameter} and the \emph{separation} of $T$ respectively by
\[
\spt(T)=\{p_1,\dotsc,p_Q\}
\,,\quad
\diam(T):=\max_{i,j}|p_i-p_j|
\quad\text{and}\quad
\sep(T):=\min\{|p_i-p_j|\,:\,p_i\neq p_j\}\,.
\]
Moreover given $T=\sum_{i=1}^Q\a{p_i}$ and a vector $v\in \R^n$, we define
\[
T\oplus v:=\sum_{i=1}^Q\a{p_i+v}
\qquad \text{and}\qquad
T\ominus v:=\sum_{i=1}^Q\a{p_i-v}\,.
\]
We define the space of \emph{average free $Q$-points} by
\[
\Iqsn:=\left\{T\in \Iq(\R^n)\,:\, \eta\circ T=0 \right\}\,,
\]
where $\eta\circ T:=\frac1Q \sum_ip_i$ is the \emph{average} of $T=\sum_{i=1}^Q\a{p_i}$.

Any map $f\colon \Omega \to \Iqs$ will be called a \emph{$Q$-valued map}. We refer the reader to \cite{DS} for the Sobolev theory of $Q$-valued maps. In particular, we observe that \cite[Chapter 4]{DS} provides such  a theory independently of Almgren's embedding. Finally the derivation of \eqref{eq:outer} and \eqref{eq:inner} are contained in \cite[Section 3.1]{DS}.

In the sequel we will often use the following useful observations concerning the invariance of outer and inner variation under shift by harmonic functions.

\begin{lemma}\label{lem.OuterInvariance}
	Let $h: \Omega \to \R^n$ be a single-valued harmonic map then for any $Q$-valued map $f \in W^{1,2}(\Omega, \Iq)$ we have 
	\[ \mathcal O(f\oplus h,\psi_h) = \mathcal O(f, \psi)\;, \]
	where $\psi_h(x,z)=\psi(x, z-h(x))$.
\end{lemma}
\begin{proof}
	It is sufficient to show that for any smooth single-valued map $h$ one has 
	\[O(f\oplus h,\psi_h) = \mathcal O(f, \psi) + Q \int_\Omega \langle Dh \, : \, D (\eta \circ\psi(x,f)) \rangle \, dx\,.\]
	Since $D_x\psi_h(x,z)= (D_x\psi)(x,z-h) - D_y\psi(x,z-h)\cdot Dh$ we have 
	\[D_x\psi_h(x,f+h) + D_y\psi_h(x,f+h)\cdot (Df + Dh)(x) = (D_x\psi)(x,f) + (D_y\psi)(x,f)Df(x) \,,  \]
	from which the above formula follows. 
\end{proof}

\begin{remark}
If 	$f \in W^{1,2}(\Omega, \Iq)$ is stationary with respect to the outer variation then its average $\eta \circ f$ is a single-valued harmonic function. 
This can be seen by using a variations $\psi(x,z)= \tilde{\psi}(x)$ that are independent of $z$. In particular we deduce that $\eta \circ f $ is a smooth function. This had been already observed in \cite[Lemma 3.23 (a)]{DS}.
\end{remark}

\begin{lemma}\label{lem.InnerInvariance}
	Let $h: \Omega \to \R^n$ be a single-valued harmonic map and $f \in W^{1,2}(\Omega, \Iq)$ stationary with respect to the outer variations then we have 
	\[\mathcal I(f \oplus h, \phi) = \mathcal I(f, \phi) + Q\, \mathcal I(h, \phi)\,.\]
\end{lemma}
\begin{proof}
	Expanding all quadratic expressions $\partial_i (f+h) \partial_j (f+h)$ leads to 
	\begin{align*}
		\mathcal I(f \oplus h, \phi) =& \mathcal I(f, \phi) + Q\, \mathcal I(h, \phi)\\&+ 2Q \left( \int \langle Dh \colon D\eta\circ f \cdot D\phi \rangle + \langle D\eta \circ f \colon Dh \cdot D\phi\rangle - \langle Dh \colon D\eta \circ f\rangle \Div(\phi) \right)\,.
	\end{align*}
	We have to show that the last integral is vanishing for the two harmonic functions $h$ and $g=\eta \circ f$. This follows by integrating by parts and using the following identity
	\begin{align*}
		&\langle Dh \colon Dg\cdot D\phi\rangle + \langle Dg \colon Dh\cdot D\phi\rangle\\& = \sum_{ij} \langle \partial_ih, \partial_i ( \partial_j g \phi^j)\rangle + \langle \partial_ig, \partial_i ( \partial_j h \phi^j)\rangle - (\langle \partial_ih, \partial_{ij}g\rangle + \langle \partial_ig, \partial_{ij}h\rangle)\phi^j\\
		&= \sum_{ij} \langle \partial_ih, \partial_i ( \partial_j g \phi^j)\rangle + \langle \partial_ig, \partial_i ( \partial_j h \phi^j)\rangle - \partial_j (\langle \partial_ih, \partial_ig\rangle )\phi^j\,.
	\end{align*}
	
\end{proof}

In particular we deduce that adding an harmonic function $h$ to a stationary map $f$ will preserve stationariety.

\subsection[Splitting Q-points]{Splitting $Q$-points}

Given a point $S\in \Iq(\R^n)$, with $S=\sum_{i=1}^q Q_i\a{s_i}$, $s_i\neq s_j$ for $i\neq j$, we define
\[
r_i(S)=r_i:=\min\{\dist(s_i, s_j)\,:\,j\neq i\}\,,\qquad i=1,\dots,q\,,
\]
and we define the \emph{projection neighborhoods of $S$} by
\[
\mathcal P_\eps(S):=\bigcup_{i=1}^q B_{\eps r_i}(s_i)\,,\qquad 0<\eps<\frac14\,.
\]
We will associate to such a point $S$ a Lipschitz retraction $P_S: \R^n \to \mathcal P_{\frac18}(S)$ such that $P_S(t) = \eta(S)$ if $ t \notin \mathcal{P}_{\frac14}(S)$ with $\operatorname{Lip}(P_S) \le C \,\sep(S)^{-1} $.
\begin{remark}[Balanced splitting]\label{rem:natsplit} Notice that the map $P_S$ that projects each ball $B_{\eps r_i}(s_i)$ on the center $s_i$ is a well defined Lipschitz retraction, and moreover if $T\in \Iq(\R^n)$ is such that 
\[
\spt(T)\subset \mathcal P_\eps(S)
\]
then we have a natural splitting of $T=\sum_{i=1}^q T_i$, where
\begin{equation}\label{eq:naturalsplit}
\spt(T_i)\subset B_{\eps r_i}(s_i)\,,\qquad \forall i=1\dots q\,.
\end{equation}
Using the notation of currents we have $T_i = T\res B_{\eps r_i}(s_i)$. 
We will call such a $T$ \emph{balanced} if the mass of $T$ inside the ball $B_{\epsilon r_i}(s_i)$ coincides with the mass of $S$ inside the same ball i.e. $\M(T_i)= Q_i= \M(T\res B_{\eps r_i}(s_i))$ for all $i$.

Furthermore note that if $T \in \mathcal P_\epsilon(S)$ is balanced and $\epsilon <\frac14$ then
\begin{equation}\label{eq:splitsum}\G^2(S,T)=\sum_{i=1}^q \G^2(T_i,Q_i \a{s_i})\,.
\end{equation}
This follows from the pigeonhole property. 
Indeed let $T= \sum_{j=1}^Q \a{t_j}$ and 
let $(I_i)_{i=1}^q$ be the partition of $\{1, \dotsc , Q\}$ given by $I_i=\{ \sum_{j<i} Q_j +1, \dotsc, \sum_{j\le i} Q_j\}$ and select a permutation $\sigma \in \mathcal{P}_Q$ that realises the distance, i.e. 
\[\G^2(S,T) = \sum_{i=1}^q \sum_{j\in I_i} |s_i-t_{\sigma(j)}|^2\,.\]

Assume by contradiction it is not true that $\sigma(I_i)=I_i$, then by the pigeonhole property there must be a ``k-chain'' (after relabeling the $I_i$'s) of indices $j_1, \dotsc, j_k$ with the property that 
\begin{enumerate}
    \item $j_i \in I_i$;
    \item $\sigma(j_i) \in I_{i+1}$ for $i=1, \dotsc, k-1$ and $\sigma(j_k)\in I_1$.
\end{enumerate}
Then we have on the one hand
\begin{align*}
    \sum_{i=1}^k |s_{i}-t_{\sigma(j_i)}|^2
    \geq (1-\epsilon)^2 \sum_{i=1}^k r_i^2\,,
\end{align*}
but on the other hand we have as a competitor
\[ \sum_{i=1}^k |s_i-t_{j_i}|^2 \le \epsilon^2 \sum_{i=1}^k r_i^2\,.\]
This contradicts the optimality of $\sigma$.

With a similar proof one can show that if $\spt S, \spt S'\in \mathcal P_\epsilon(T)$ are balanced and $\epsilon <\frac18$ then
\begin{equation}\label{eq:splitsum2}\G^2(S,S')=\sum_{i=1}^q \G^2(S_i,S'_i )\,\,.
\end{equation}

Additionally one can construct, compare \cite[Lemma 3.7]{DS}, Lipschitz retractions \[\chi_i \colon \Iqs \to B_{2r_i}\left(Q_i\a{s_i}\right)\subset \A_{Q_i}(\R^n)\,.
\]
such that for all $T \in \mathcal{P}_\epsilon(S)$ balanced one has $T= \sum_{i=1}^Q \chi_i(T)$, 
since $\chi_i(T)= T_i$ defined above. In particular this implies that if $f\in W^{1,p}(\Omega, \Iq)$ has the property that $f(x) \in \mathcal{P}_\epsilon(S)$, balanced for a.e. $x$, then $f$ splits in the following sense: setting $f_i:=\chi_i\circ f \in W^{1,p}(\Omega, \I{Q_i})$ one has $f=\sum_{i=1}^q f_i$ and $|Df(x)|^2= \sum_{i=1}^qe |Df_i(x)|^2$ for a.e. $x$.

Let us note that if we apply this observation to the setting of stationary maps one deduce from the fact that the balls $B_{\epsilon r_i}(s_i)$ are disjoint and the outer variation localises that 
if $f\in W^{1,2}(\Omega, \Iq)$ takes values a.e. in $\mathcal{P}_\epsilon(S)$ then 
\begin{equation}
    \mathcal{O}(f,\tilde{\psi})= \sum_{i=1}^q\mathcal{O}(f_i,\psi_i)
\end{equation}
where $\tilde{\psi}= \sum_{i=1}^Q \theta_i(z)\psi_i(x,z) $ where $\theta_i=1$ on $B_{\epsilon r_i}(s_i)$ and supported in $B_{2\epsilon r_i}(s_i)$.

\end{remark}
Next we recall the following splitting lemma from \cite{DS1}. 

\begin{lemma}[{\cite[Lemma 3.8]{DS1}}]\label{lem:camilloemanuele}
For every $0 < \eps < 1$, we set $\beta(\eps, Q) = (\frac{\eps}{3})^{3^Q}$ . 
Then, for every $T \in \Iq$ there exists a point $S\in \Iq$ such that
\begin{gather} 
\beta(\eps, Q) \, \operatorname{diam}(T ) \leq \sep (S) < \infty,\label{eq:sep1} \\
\cG(S, T ) \leq \eps \,\sep (S)\,,\label{eq:sep2}\\
\spt(S) \subset \spt(T)\,.\label{eq:sep3}
\end{gather}
\end{lemma}

Next we state and prove a combinatoric lemma that will be used later in the proof of the continuity of stationary maps.
Roughly speaking the following lemma provides an iterative scheme that detects regions of ``comparable'' separation and applies there Lemma \ref{lem:camilloemanuele}. In this manner we are able to construct a family of $Q$ points $S_k$, which have the strong property that the separation of the ``finer'' one controls the distance to the next ``coarser'' one. We will need this to control the jumps when re-centering the monotonicity formula of Proposition \ref{prop:energymon} below.

\begin{lemma}[Key combinatoric lemma]\label{lem:key}
Let $0<\eps<1/8$ and set $\tilde{\beta}(\eps, Q) = \left(\eps + \frac{Q-1}{\beta(\epsilon, Q)} \right)^{-1}$, with $\beta(\eps, Q)$ as in Lemma \ref{lem:camilloemanuele}. Then for every $T\in \Iq$ there exist $q\leq Q$ and a family of points $\{S_i \in \Iq\,:\, i=1,\dotsc q\}$ such that $S_0=Q\a{t}$, $S_q=T$, and, for every $k=0,\dotsc,q$, the following properties hold:
\begin{enumerate}
    \item $\spt(S_k) \subset \spt(T)$;
    \item $\spt (T) \subset \mathcal{P}_\eps(S_k)$ and it is balanced, as defined in Remark \ref{rem:natsplit};
    \item $\tilde{\beta}(\epsilon, Q)^k \,\G(S_{k-1}, S_k) \le 2\,\sep(S_k)$\,.
\end{enumerate}
\end{lemma}

\begin{proof} 
We will prove the lemma replacing (3) with the stronger conclusion
\begin{enumerate}
    \item[(3)*] $\tilde{\beta}_k\,\max\left( \G(T,S_k), \G(T, S_{k-1})\right) \le  \sep(S_k) $,
\end{enumerate}
where $\tilde{\beta}_k=\tilde{\beta}(\epsilon, Q)^k$. Notice that indeed (3) follows from (3)* and the triangular inequality. 

We prove the lemma by induction. At several steps in the proof we will use the following inequality for $T\in \Iq$ and $p\in \spt (T)$, which follows directly from the definition of $\G$: \[
\frac{1}{\sqrt{2}} \diam(T) \le \cG(T, Q \a{p}) \le \sqrt{Q-1} \,\diam(T)\,.
\]

For the base step we let $t\in \spt (T)$ and set $S_0:=Q\a{t}$. Next we apply Lemma \ref{lem:camilloemanuele} to $T$ and find a $Q$-point $S_1=S\in \Iq$ satisfying \eqref{eq:sep1}, \eqref{eq:sep2} and \eqref{eq:sep3}. In particular (1) follows from \eqref{eq:sep3}, (2) follows from \eqref{eq:sep2}, since $\sep(S)\leq r_i(S)$ for every $i$, giving that $\spt T\in \mathcal{P}_\epsilon(S_1)$. Furthermore $T$ needs to be balanced, otherwise  $\G(T,S_1)\ge (1-\epsilon)\sep(S_1)$, contradicting \eqref{eq:sep2}. Finally (3)* follows since from \eqref{eq:sep2} we have $\G(S_1,T)\le \eps\, \sep(S_1)$ and, from \eqref{eq:sep2} and \eqref{eq:sep1} we have
\[
\frac{\beta(\epsilon,Q)}{\sqrt{Q-1}} \,\G(T,S_0)\leq \beta(\epsilon,Q) \, \diam(T)\leq \sep(S_1)\,.
\]

Next assume we have $S_{k-1}=\sum_{i=1}^l Q_i \a{s_i}$, $l\leq Q$, satisfying (1), (2) and (3)*. By Remark \ref{rem:natsplit}
with $T$ and $S=S_{k-1}$, we can write 
\[
T=\sum_{i=1}^{l}T_i
\]
with $T$ balanced. Without loss of generality we can assume that
\[
\cG(T_1,Q_1\a{s_1})\geq\cG(T_{2},Q_{2}\a{s_{2}})\ge \dotsc \ge \G(T_l, Q_l \a{s_{l}})\,. 
\]
Next we apply Lemma \ref{lem:camilloemanuele} with $T=T_1$ to get $S'=\sum_{j=1}^m Q_j' \a{s_j'} \in \I{Q_1}$,  with $\sum_{j=1}^m Q_j'= Q_1$, and we set 
\[
S_{k}:= S'+\sum_{i=2}^l Q_i \a{s_i}\in \Iq\,.
\]
Then (1) follows exactly as in the base step as a consequence of \eqref{eq:sep3}.
To see (2) we note first that since $\spt(S')\subset \spt(T_1)$ we have \[\diam(S')\le \diam(T_1) \le 2\epsilon r_1(S_{k-1}) = 2\epsilon r_1\]
where we have used the induction hypothesis (2) i.e. $T \in \mathcal{P}_\epsilon(S_{k-1})$. 
This implies that for any $s' \in S'$ and $j>1$
\[ |s'-s_j| \ge |s_1-s_j| - |s'-s_1| \ge (1-\epsilon)r_1\,.\]
From the previous two displayed inequalities we deduce that 
\[\sep(S') \le 2\epsilon r_1 \le (1-\epsilon) r_1 \le |s' - s_j|\,,\qquad \forall s' \in \spt(S'),\,  j>1\,.\]
Since $\G(T_1,S')\le \epsilon \,\sep(S')$, we have 
\[T \in \mathcal{P}_\epsilon(S_k)\,.\]
To prove (3)* we note that the above implies that 
\begin{equation}\label{eq:sep_k}\sep(S_k) = \min\left\{\sep(S') , \min_{\substack{
i,j > 1\\i\neq j }}|s_i-s_j|\right\}{\ge} \min\{ \sep(S'), \sep(S_{k-1})\} \,.\end{equation}
The inductive assumption (3)* for $S_{k-1}$ in the 
 last inequality gives 
\begin{equation}\label{eq:annoying1}
\tilde{\beta}_{k-1}\, \sep(S') \le \tilde{\beta}_{k-1} \, \G(T_1, Q_1\a{s_1})\le  \tilde{\beta}_{k-1} \, \G(T, S_{k-1})  \le \sep(S_{k-1})\,.
\end{equation}
Combining the trivial inequality $\tilde{\beta}_{k-1}\, \sep(S')\le \sep(S')$ with \eqref{eq:annoying1}, \eqref{eq:sep_k} gives
\begin{equation}\label{eq:inductionhypo}
   \tilde{\beta}_{k-1} \,\sep(S') \le \sep(S_k) \le \sep(S') \,.
\end{equation}
Next notice that
\begin{equation}\label{eq:inductionstep1}
    \frac{\beta(\epsilon, Q_1)}{\sqrt{l}\sqrt{Q_1-1}} \G(T,S_{k-1}) \le\frac{\beta(\epsilon, Q_1)}{\sqrt{Q_1-1}} \G(T_1,Q_1\a{s_1}) \le \beta(\epsilon, Q_1)\, \diam(T_1)\leq \sep(S')\,,
\end{equation}
where the last inequality follows from \eqref{eq:sep1}. Moreover 
\begin{equation}\label{eq:inductionstep2}
\G(T, S_k)\le \G(T_1,S_1) + \G(T,S_{k-1}) \le \left(\eps + \frac{\sqrt{l}\,\sqrt{Q_1-1}}{\beta(\epsilon, Q_1)}\right)\, \sep(S') \,.
\end{equation}
The conclusion now follows combining \eqref{eq:inductionhypo}, \eqref{eq:inductionstep1} and \eqref{eq:inductionstep2}, with
\[\tilde{\beta}_k^{-1} \ge \tilde{\beta}_{k-1}^{-1} \; \left(\eps + \frac{Q-1}{\beta(\epsilon, Q)}\right) \geq \left(\eps + \frac{Q-1}{\beta(\epsilon, Q)} \right)^k\,.\]

\end{proof} 

\section{Average conformal maps and monotonicity of the energy}

In this section we introduce average conformal maps and we prove that every stationary map with domain in $\R^2$ can be modified to obtain an average conformal stationary map. This will follow from the inner variation formula. The relation between inner variation for the energy of an harmonic map and conformality is well known, see for instance \cite{Schoen}.

\subsection{Average conformal maps} We introduce the following class of multivalued maps. 

\begin{definition} We say that a map $f\in W^{1,2}(B_1\subset\R^2;\Iq(\R^n))$ is \emph{average conformal} if
\begin{equation}\label{eq:avg_conf}
\sum_{i=1}^Q |D_1f_i|^2(x)=\sum_{i=1}^Q |D_2f_i|^2(x)
\qquad \text{and}\qquad
\sum_{i=1}^Q D_1f_i(x)\cdot D_2f_i(x)=0\,,
\end{equation}
for almost every $x\in B_1$.
\end{definition}

Next we consider the following \emph{modified energy density} on a set $\Omega$
\[
\Theta_f(\Omega, S, s):=\frac1{s^2}\int_{\Omega} \phi\left(\frac{\cG(f, S)}{s}\right) |Df|^2\,dx\,,
\]
where $\phi(t)$ is a smooth non-negative, non-increasing function, with $\phi=1$ on $t<\frac12$ and $\phi=0$ for $t\ge 1$.
A nice property of average conformal maps is the following monotonicity of the energy density (notice that the domains in the integral are not balls in the domain of $f$, but in the target).

\begin{proposition}[Energy monotonicity]\label{prop:energymon} Let $f\in W^{1,2}(\R^2;\Iq(\R^n))$ be an average conformal map, then $s\mapsto \theta_f(\Omega,S,s)$ is monotone non decreasing for $s\in [0,r_S]$, where $r_S>0$ satisfies 
\begin{enumerate}
    \item $r_S\le \frac14 \sep(S)$;
    \item $\{\G(f,S)<r_S\} \subset \Omega$.
\end{enumerate}
\end{proposition}

\begin{proof}
Let $s\in [0, r_S]$ and let $P_S$ be the Lipschitz retraction of Remark \ref{rem:natsplit}. We test the outer variation \eqref{eq:outer} with the vector field 
\[
\psi(x,z)= \phi\left(\frac{\G(f(x),S)}{s}\right) (z- P_S(z))\,,
\]
which is an admissible vector field by approximation and since it is $0$ on the boundary of $\Omega$ due to the assumptions on $r_S$.
Hence we have 
\begin{equation}\label{eq:out_mon}
    \int_{\Omega}\phi\left(\frac{\cG(f, S)}{s}\right) |Df|^2\,dx + \int_{\Omega} \phi'\left(\frac{\cG(f, S)}{s}\right) \,\frac{\cG(f,S)}{ s} \,|D \G(f,S)|^2= 0\,,
\end{equation}
where we used that for a.e.  $ x \in \left\{\cG(f, S) < \frac14 \sep(S)\right\}$ the following holds
\begin{equation}\label{eq:id1}
\sum_{i=1}^Q Df_i(x)\colon (f_i(x)- P_S(f_i(x)) )= \G(f(x),S)\, D \G(f(x),S)\,.
\end{equation}
To justify the above calculations one can use a sequence of Lipschitz function $\tilde{f}^k$ that approximate $f$ strongly in $W^{1,2}$. Passing to $f^k=P_S\circ \tilde{f}^k$, we can assume that $ \G(f^k(x),S) \le \frac14 \sep(S)$ for all $x$, and the Lipschitz functions $f^k$ still approximate $f$ strongly in $W^{1,2}$ on the set $\{\G(f,S)<r_S\}$.

Next notice that, since $\phi'<0$ we have
\begin{align*}
\frac{d}{ds} \Theta_f(\Omega,S,s)
    &=-\frac{1}{s^3}\left( \int_\Omega \phi'\left(\frac{\cG(f, S)}{s}\right) \,\frac{\cG(f,S)}{s} \,|Df|^2+2 \int_\Omega \phi\left(\frac{\cG(f, S)}{s}\right) \,|Df|^2\right)\\
    &=-\frac{1}{s^3}\int_\Omega \phi'\left(\frac{\cG(f, S)}{s}\right) \,\frac{\cG(f,S)}{s} \,\left( |Df|^2-2|D\G(f,S)|^2\right)\ge 0\,,
\end{align*}
where the last inequality follows from the fact that $f$ is averaged conformal in the following way. Let $x$ be a Lebesgue point of $f$, and let $f_i$ be a selection such that $Df= \sum_{i=1}^Q \a{Df_i}$ and $\G(f,S)^2 = \sum_{i=1}^Q |f_i(x)- s_{i}|^2 $, where $s_i \in \spt(S)$. Define the $\R^{Qn}$ vectors
\begin{align*}
e_i&=\frac{(\partial_if_1(x), \partial_if_2(x), \dotsc, \partial_if_Q(x))}{\sqrt{\frac{|Df(x)|^2}{2}}}\,,\qquad i=1,2\\  b&= (f_1(x)-s_1, f_2(x)-s_2, \dotsc, f_Q(x)-s_Q)\,.
\end{align*}
Average conformality implies that $|e_1|^2=|e_2|^2=1, \langle e_1, e_2\rangle =0$, hence $(e_i)_{i=1}^2$ are orthonormal.  Therefore from \eqref{eq:id1} we have 
\begin{align*}
    \G(f(x),S)^2|D\G(f(x),S)|^2 & = \frac12 |Df(x)|^2 \,\sum_{i=1}^2 \langle e_i, b\rangle^2\\
    &\le \frac12 |Df(x)|^2\, |b|^2 = \frac12 \G(f(x),S)^2\, |Df(x)|^2\,,
\end{align*}
which concludes the proof.
\end{proof}

We will also need the following: 

\begin{lemma}[Energy lower bound]\label{lem:lowbound} Let $f\in W^{1,2}(B_2, \Iq)$ and $\Omega \subset B_2$, then 
\begin{equation}\label{eq:lower_ener_bd}
    \lim_{s\to 0}\Theta_f(\Omega,f(x_0),s)\geq \frac12\int \phi(x)\,dx>0 \qquad \text{for a.e. }x_0\in \Omega\setminus \{|Df|=0\}\,.
\end{equation}
\end{lemma}

\begin{proof}
Let us fix a point $x_0 \in \Omega$ with the properties that
\begin{enumerate}
    \item $x_0$ is a full density point of $\Omega$,
    \item $f$ is approximately differentiable in $x_0$,
    \item $x_0$ is a Lebesgue point of $|Df(x)|^2$.
\end{enumerate}
These three properties hold for a.e. point in $\Omega$ by the Lebesgue differentiation theorem and \cite[Corollary 2.7]{DS}.

After translating we may assume $x_0=0$ and $d^2= |Df(0)|^2$. Let $Tf(x)=\sum_{i=1}^Q \a{f_i(0) + Df_i(0)\cdot x}$ the first order approximation of $f$ in $0$. In particular (1), (2) and (3) read
\begin{enumerate}
    \item $\lim_{r\to 0}r^{-2}|\Omega \cap B_r|=\omega_2$
    \item for all $\lambda>0$ one has $ \lim_{r\to 0} r^{-2} | \{x \colon \cG(f(x), Tf(x)) \ge \lambda |x| \} \cap B_r| =0$;
    \item $\lim_{r\to 0} r^{-2} \int_{B_r} | |Df(x)|^2-d^2| \,dx= 0 $.
\end{enumerate}
Let $0<\epsilon<\frac14$ be given and set $U=\{x \colon \cG(f(x), Tf(x)) \ge \epsilon d |x| \}$, so that by (1)
\begin{equation}\label{eq:lb1}
\lim_{r\to 0} r^{-2} | U \cap B_r| =0\,.
\end{equation}
Then for each $x \notin U$ we have
\begin{equation}\label{eq:lb2}
\G(f(x),f(0))\le \G(f(x),Tf(x)) + \G(Tf(x),f(0)) \le (1+\epsilon) \, d\,|x|,
\end{equation}
where we used that $\G(Tf(x),f(0))^2 \le \sum_{i=1}^Q |Df_i(0)\,x|^2 = d^2 |x|^2$. 
Hence we have for $s=(1+\epsilon)d r$ 
\begin{align*}
   &s^{-2} \int_{\Omega} \phi\left(\frac{\cG(f, f(0))}{s}\right) |Df|^2\,dx \ge s^{-2} \int_{\Omega\cap B_r\setminus U} \phi\left(\frac{\cG(f, f(0))}{s}\right) |Df|^2\,dx \\&\ge s^{-2} \int_{\Omega \cap B_r\setminus U} \phi\left(\frac{\cG(f, f(0))}{s}\right) d^2\,dx - \frac{2d^{-2}}{r^2} \int_{B_r} ||Df|^2-d^2| \,dx\\
   &\ge  ((1+\epsilon)r)^{-2} \int_{\Omega \cap B_r\setminus U} \phi\left(\frac{|x|}{r}\right) - o(1) \ge ((1+\epsilon)r)^{-2} \int_{B_r} \phi\left(\frac{|x|}{r}\right) - o(1) \\
   &= (1+\epsilon)^{-2}\int \phi(x) \,dx - o(1),
\end{align*}
where in the next to last inequality we used (3), in the last inequality we used \eqref{eq:lb2} and (1), and in the last equality we used \eqref{eq:lb1}. 
\end{proof}
\subsection{2-d stationary maps are average conformal}

Using an analogue of the ``Hopf differential'' we show that $2$-dimensional stationary maps can be modified to average conformal stationary maps. Similar ideas have been used in \cite[Theorem 3.2]{Schoen}, \cite{Grueter2} and \cite[Proposition 1]{Lin}, in chronological order. We define the \emph{average Hopf differential} of a map $f\colon \C \to \Iq$ by
\[
\sum_{i=1}^Q\partial_z f_i\,\,\partial_z f_i=\sum_{i=1}^Q|D_1f_i(z)|^2-|D_2f_i(z)|^2 - 2 i D_1f_i(z)\cdot D_2f_i(z)\,,
\]
and we observe that a map is average conformal if its average Hopf differential is zero.

\begin{proposition}[Stationarity implies conformality]\label{prop:statavgconf}
Let $f\in W^{1,2}(B_1, \Iq)$ be stationary with respect to the Dirichlet energy, then there exists an harmonic function $v \in C^\infty(B_1, \R^2)$ such that 
\begin{equation}\label{eq:madeconformal} F= f \oplus (v_1 e_{n+1} + v_2 e_{n+2}) = \sum_{i=1}^Q \a{ (f_i, v)} \in W^{1,2}(B_1, \Iq(\R^{n+2}))\,.\end{equation}
is average conformal. In particular $F$ is stationary with respect to the Dirichlet energy and moreover
\begin{equation}\label{eq:conformalenergy}
    \int_{B_r} |DF|^2\,dx\leq \int_{B_r}|Df|^2\,dx+ C_Q \, r^2 \, \int_{B_1} |Df|^2\,dx \quad \forall B_r \subset B_{\frac45}\,.
\end{equation}
\end{proposition}

\begin{proof} 
We repeat the proof suggested by Gr\"uter in \cite{Grueter2} with a modification to control the energy. It is well-known that the inner variation can be written as 
\[
\int \Phi(z) \,\partial_{\bar{z}} \eta  = 0\,,\qquad \forall \eta \in C^\infty_0(\Omega, \C)\,, \footnote{Direct computation shows that \[\sum_{i,j} \left(2 D_if\cdot D_jf - \delta_{ij} |Df|^2\right) D_i\phi^j = \operatorname{Re}\left(\Phi_f(z)\partial_{\bar{z}}\phi\right)\]
where we have set $\Phi_f(z)=(\partial_z f(z))^2 =|D_1f(z)|^2-|D_2f(z)|^2 - 2 i D_1f(z)\cdot D_2f(z) $ and used the classical identification of $\R^2$ with $\C$ such that $\phi\equiv \phi_1 + i \phi_2$\,. Now choosing $\phi$ and $i\phi$ one obtains the claimed equation.}
\]
where $\Phi(z)= \sum_{i=1}^Q |D_1f_i(z)|^2-|D_2f_i(z)|^2 - 2 i D_1f_i(z)\cdot D_2f_i(z)$ is the average Hopf differential of $f$. 
In particular this implies that $\Phi(z)$ is holomorphic. Then let $\Psi$ be a primitive to $\Phi$ and set $E^2=\int_{\Omega} |Df|^2$ to be the Dirichlet energy of $f$, and define the function 
\[v=v_1+ iv_2 = \frac{1}{2\sqrt{Q}}E \bar{z} - \frac1{2\sqrt{Q}} \frac1E\Psi\,.
\]
Notice that $v$ is harmonic, since $\Phi$ is holomorphic. Reasoning as in \cite{Grueter2}, the Hopf differential of $v$ is given by 
\[ 
\partial_z v \,\partial_z \bar{v} = \partial_z v\, \overline{ \partial_{\bar{z}} v} = - \frac1{4Q} \Phi.
\]
To obtain \eqref{eq:conformalenergy} notice that, for any $\Omega' \Subset \Omega$, its Dirichlet energy is 
\[ \int_{\Omega'} |\partial_z v|^2 + |\partial_{\bar{z}}v|^2 = \frac{E^2}{4Q}|\Omega| + \frac{1}{4QE^2} \int_{\Omega'} |\Phi|^2 \le  \frac{E^2}{4Q}|\Omega| + \frac{C}{4QE^2} \left( \int_{\Omega} |\Phi| \right)^2 \le C E^2\,,  \]
where we used regularity for holomorphic functions, and the constant $C$ depends only on $\Omega$ and $\Omega'$. Finally we define $F$ as in \eqref{eq:madeconformal}, and we notice that its averaged Hopf differential is vanishing pointwise since
\[ \Phi_F(z) = \Phi(z) + 4Q \partial_z v \,\partial_z \bar{v} =0\,, \]
so that $F$ is averaged conformal.

\end{proof}

\section{Continuity, Higher integrability and H\"older regularity}

In the first subsection we prove continuity of $2$-d stationary maps. We then use it to prove higher integrability, while H\"older regularity is a straightforward corollary of these two results.

\subsection{Continuity: proof of Theorem \ref{thm:cont}}

By Proposition \ref{prop:statavgconf} we can assume that $f=F$ is average conformal and still stationary.

Let $y_0\in B_{1/2}$ be fixed and let $0<r<1/4$. Replacing $f(\cdot)$ with $f(y_0+r\cdot)$ we can assume that $r=1$ and $y_0=0$. By Fubini, there exists $\rho\in (1,\frac32)$ such that
\[
\int_{\partial B_\rho}|Df|^2\leq \int_{B_{\frac32 }}|Df|^2\,dx=:E_0
\]
In particular the embedding of $W^{1,2}(\partial B_\rho, \A_Q) \hookrightarrow  C^{0,\frac12}(\partial B_\rho)$,  \cite[Proposition 1.2 (a)]{DS},  applied to $f|_{\partial B_\rho}$ gives, together with the above estimate gives,
\begin{equation}\label{eq:Hoelderonboundary}
\G(f(x),f(y))^2 \le C_1 \left(\int_{B_{\frac32}} |Df|^2 \,dx\right) |x-y| \,,  \qquad \forall x,y \in \partial B_\rho\,.
\end{equation}

\medskip

\noindent\emph{Claim 1:} Let $x_0 \in B_\rho$ be as in Lemma \ref{lem:lowbound}, with $\Omega=B_\rho$, then we claim that 
\begin{equation}\label{eq:claim1}
    \inf_{y \in \partial B_\rho} \G(f(x_0), f(y))^2 \le C_{Q} \int_{B_{\frac32}} |Df|^2 \, dx\,.
\end{equation}

Indeed, let us denote with $r^2_0$ the value on the left hand side. We may assume that $r^2_0> 0$ otherwise there is nothing to show. Notice that $\G(f(x), f(x_0))\ge r_0 $ for all $x \in \partial B_\rho$, so that 
\[
\{x\in B_\rho\,:\,\G(f(x),f(x_0)) < r_0\} \subset B_\rho\,.
\] 
Hence we have combining Lemma \ref{lem:lowbound} with Proposition \ref{prop:energymon} that 
\[ c\le \Theta_f(B_\rho, f(x_0), s_0)\leq\frac{1}{s_0^2} \int_{B_\rho} |Df|^2 \,,\qquad s_0= \min\{r_0, \frac14\sep(f(x_0))\}.\]

If $s_0 = r_0$ the claim follows. 

If not we make use of the combinatoric Lemma \ref{lem:key} as follows. Let $S_0=Q\a{t}, S_1, \dotsc, S_{q-1}, S_q=T=f(x_0)$ be the sequence of points constructed there with parameter $\epsilon = \frac18$.
Note that we have by property (3) of Lemma \ref{lem:key}
\[ 
\left\{ P \colon \G(P, S_k) \le  s \right\} \subset \left\{ P \colon \G(P, S_{k-1}) \le \frac12 M s \right\}\,,\quad \text{ for  any }s> \frac18 \sep(S_k)\,,
\]
where $\frac12 M=1 + 16 \tilde{\beta}(\frac18,Q)^{-Q}$,  and therefore 
\begin{equation}\label{eq:energyjump} \Theta_f(B_\rho, S_k, s) \le M^{2} \Theta_f(B_\rho, S_{k-1}, Ms)\,, \quad \text{ for  any }s> \frac18 \sep(S_k)\,.
\end{equation}

Let us define the sequence of ``stopping times'' $b_k$ by setting $b_q = \frac14 \sep(f(x_0))$, and then inductively backwards for $k< q$ by
\[ 
b_k =\max\left\{ Mb_{k+1},  \min\left\{ \frac14 \sep(S_k) , s_k \right\}\right\}\,,\quad s_k:=\min_{y \in \partial B_\rho} \G(f(y), S_k)\,.
\]
 Furthermore let $k^*-1$ be the last $k$ such that $b_k < s_k$, so that $b_{k*} \ge s_{k^*}=\G(f(y_*), S_k)$, for some $ y_* \in \partial B_\rho$.

Now we will show by backwards induction that for each $k_* \le k \le q$ we have
\begin{equation}\label{eq:energyjumpmon}
M^{2(k-q)} c^2 \le \Theta_f(B_\rho, S_k, s) \,,\quad\text{ for } M b_{k+1} \le s \le b_k\,. 
\end{equation}
This is obtained combining \eqref{eq:energyjump} with the the energy monotonicity, Proposition \ref{prop:energymon}. Note that \eqref{eq:energyjumpmon} implies that 
\[ 
b_k^2 M^{2(k-q)}c^2 \le  \int_{B_\rho} |Df|^2 \,dx\,,\qquad \forall k\ge k^*\,,
\]
where we used that
\[
\{x\in B_\rho\,:\,\G(f, S_k)< s_k\}\subset B_\rho\,.
\]
Using (3) of Lemma \ref{lem:key} and the fact that $\sep(S_k)\leq 4 b_k$, for $k^*-1\leq k\leq q$, by the choice of $k^*$, we then conclude 
\begin{align*}
\G(f(x_0), f(y_*)) &\le \sum_{l \ge k_*+1}^q \G(S_l, S_{l-1}) + \G(f(y_*), S_{k^*}) \\
& \le  \frac{4}{\tilde\beta(\eps,Q)^Q} \sum_{l \ge k^*}^q b_k \le C \left( \int_{B_\rho} |Df|^2 \, dx \right)^\frac12\,.
\end{align*}

\medskip

\emph{Claim 2: } For a.e. $x \in B_\rho$ we have \eqref{eq:claim1}. 

Fix $y_1 \in \partial B_\rho$ and note that for each $x_0\in B_\rho$ as in Claim 1 or any other $x_0\in \partial B_\rho$ we have due to Claim 1 and \eqref{eq:Hoelderonboundary} that
\[ \G(f(x_0),f(y_1))^2 \le C \int_{B_\frac32} |Df|^2 \, dx = C E_0\,. \]

Consider the $W^{1,1}$ function defined on $\overline{B}_\rho$ by 
\[ g(x) = \left( \G(f(x), f(y_1))^2 - 2C E_0\right)_+\,.\]
We clearly have that $g=0$ on $\partial B_\rho$ and due to Claim 1 we have for each $x_0$ as in Lemma \ref{lem:lowbound} 
\[ Dg(x_0) = 0\,. \]
On the other hand for a.e. other point $x \in B_\rho$ we have $Df(x)=0$ hence $Dg(x)=0$, which implies $g\equiv 0$ proving the claim. 

\medskip

We can now conclude the proof of the theorem.  Indeed notice that after translating and scaling back, we have proved that 
\[ \G(f(x), f(y) )^2 \le  \int_{B_{\frac32 r}(y_0)} |Df|^2 \, dx \text{ for all } x,y \in B_r(y_0)\,.\]
Combining this with \eqref{eq:conformalenergy}, yields the conclusion.
\qed

\subsection{Higher integrability: proof of Theorem \ref{thm:highint}}\label{subsec.higherintegrability}
Theorem \ref{thm:highint} will be a consequence of a Gehring's type lemma and the following inequality. Notice the difference with \cite[Proposition 5.2]{DS1}, which is obtained with a smoothed competitor constructed via the Almgren's embedding, while our proof is independent of this embedding.

\begin{lemma}\label{lem:highint}
There is a constant $C=C(Q)>0$ such that whenever $f \in W^{1,2}(B_2, \Iq)$ is stationary with respect to outer variations and continuous in $B_{\frac{2Q-1}{Q}}$ then 
\begin{equation}\label{eq:gehring1}
    \int_{B_{1+\frac{1}{2Q}}} |Df|^2 \, dx\le C \left( \left(\int_{B_2} |Df|^2 \, dx \right)^\frac12+ \omega_f\left(2-\frac{1}{Q}\right) \right) \left( \int_{B_2} |Df| \,dx\right)
\end{equation}
where $\omega_f$ is a modulus of continuity for $f$ in $B_{\frac{2Q-1}{Q} }$, i.e. $\G(f(x),f(y)) \le \omega_f(|x-y|)$ for all $x,y \in B_{1-\frac1Q}$.
\end{lemma}

\begin{proof}
We can assume $\omega_f\left(2-1/Q\right)>0$, otherwise there is nothing to prove since $f$ would be constant. We will prove the lemma by induction on $Q$. Let us fix a radius $\frac{2Q-2}{Q} <r <\frac{2Q-1}{Q}$ such that $$\int_{\partial B_r} |Df| \le 4Q \int_{B_2} |Df| \, dx.$$

\emph{Base step:} This follows from classical regularity theory. 
Subtracting a constant we may assume that $\fint_{\partial B_r} f  = 0$, then we have from the outer variation 
\begin{align*}
    \int_{B_r} |Df|^2 &= \int_{\partial B_r} \partial_rf \cdot \left(f- \fint_{\partial B_r} f\right) \le C \, \left(\sup_{x\in \partial B_r} |f- \fint_{\partial B_r} f|\right)\,\int_{B_2} |Df|  \\
    &\le C \left( \int_{B_2} |Df|^2\right)^{\frac12}\, \int_{B_2} |Df|.
\end{align*} 

\emph{Inductive step:}
Notice that we can assume $\eta \circ f = 0 $. Indeed the average $h=\eta \circ f$ is harmonic, and so as pointed out in Lemma \ref{lem.OuterInvariance} the map $f\ominus h$ is still stationary with respect to the outer variations. Moreover it is still continuous, with modulus of continuity estimated by
\[ \omega_{f\ominus h}(r) \le \omega_f (r)+ r\left(\int_{B_2}|Df|^2\right)^\frac12\]
 since the Lipschitz constant of $h$ in $B_r$ is controlled by $\left(\int_{B_2} |Dh|^2\right)^\frac12 \le \left(\int_{B_2} |Df|^2\right)^\frac12$. Since we have the pointwise identities
\begin{align*}
    \G^2(f(x),f(y)) &= \G^2((f\ominus h)(x), (f\ominus h)(y))+ Q |h(x)-h(y)|^2\\
    |Df(x)|^2 &= |D(f\ominus h)(x)|^2 + Q |Dh(x)|^2\,,
\end{align*}
we see that the lemma is proved if it holds for $f\ominus h$, since it holds for $h$ by the case $Q=1$.

Let $M>0$ be a large number chosen later and let 
\[ D= \sup_{x\in B_r} |f(x)| = |f(x_0)|^2\,.\]
for some $x_0 \in \overline{B_r}$. 

If $D< M\, \omega_f\left(2-1/Q\right) $, we can just argue as above using the outer variations:
\[\int_{B_r} |Df|^2 = \int_{\partial B_r} \partial_rf \cdot f \le M \omega_f\left(2-\frac{1}{Q}\right) \int_{B_1} |Df|  \,.
\]

If $D> M\, \omega_f\left(2-1/Q\right)$, we apply Lemma \ref{lem:camilloemanuele} with $T=f(x_0)$ and $\eps = \frac1{16}$ and obtain a point $S$ such that 
$$\beta(\eps,Q) D \le \sep(S)
\qquad\text{and}\qquad \G(f(x_0),S) \le \frac{1}{16} \sep(S)\,.$$
Using this and triangular inequality, we have for $M$ sufficiently large
\[\G(f(x), S) \le 2\omega_f\left(2-\frac{1}{Q}\right) +  \frac{1}{16}\sep(S) \le \left(\frac{2}{M\beta} + \frac{1}{16}\right) \sep(S) \le \frac18 \sep(S)\]
Hence, since $\sep(S)>0$, as pointed out in Remark \ref{rem:natsplit} $f$ splits in $B_r$, i.e. $f= \a{f_1} + \a{f_2}$. Furthermore the remark implies that $\omega_{f_i} \le \omega_f$ for $i=1,2$. Since the outer variation splits, $f_i$ are both stationary with respect to the outer variation. Hence we can apply the lemma to $f_i$ in $B_r$ separately and obtain 
\[
\int_{B_{\left(1+\frac{1}{2(Q-1)}\right)r}} |Df_i|^2 \, dx\le \frac{C}{r} \left( \left(r^{2-n}\int_{B_r} |Df_i|^2 \, dx \right)^\frac12+ \omega_{f_i}(2-\frac1Q) \right) \left( \int_{B_r} |Df_i| \,dx\right)\,.
\]
Adding on $i=1,2$ gives the result.
\end{proof}

\begin{proof}[Proof of Theorem \ref{thm:highint}\,.]
First we prove that for any $B_{2r} \subset B_1$ we have 
\begin{equation}\label{eq:gehring} \fint_{B_r} |Df|^2 \le  C \left( \fint_{B_{2r}} |Df|^2 \right)^\frac12 \fint_{B_{2r}} |Df| \,. 
\end{equation}
Since this is scale invariant it is sufficient to prove it for $r=1$. Theorem \ref{thm:cont} implies that $f$ is continuous on $B_{2- 1/Q}$ with modulus of continuity controlled by \[\omega_f(r) \le C\left(\int_{B_2} |Df|^2\right)^\frac12(1+r) \,.\]
Hence we can apply the lemma above, Lemma \ref{lem:highint} and obtain the result.

Now Theorem \ref{thm:highint} follows from a standard modification of \cite[Theorem 6.38]{GiaquintaMartinazzi}, with equation (6.48) there replaced with \eqref{eq:gehring} i.e. 
\[ 
\fint_{B_R(x_0)} f^q \le b  \left( \fint_{B_{2R}(x_0)} f^q \right)^\frac{q-1}{q} \,\fint_{B_{2R}(x_0)} f
\]
Indeed the only place where (6.48) is used is equation (6.53), which in our case should be replaced with
\begin{align*}
    t^q
    &\leq \fint_{Q^l_{k,j} }(f\phi)^q  \le \sigma^q|P_{k,j}|^q\fint_{Q^l_{k,j} }(f)^q \le \sigma^q|P_{k,j}|^q\left(\fint_{Q^{l,(2)}_{k,j} }f\right) \left(\fint_{Q^{l,(2)}_{k,j} }f^q\right)^{q-1}\\
    &\le \sigma^{2q}  \left(\fint_{Q^{l,(2)}_{k,j} }\phi f\right) \left(\fint_{Q^{l,(2)}_{k,j} }(\phi f)^q\right)^{q-1} \le 2^n \sigma^{2q}\left(\fint_{Q^{l,(2)}_{k,j} }\phi f\right) t^{q-1}\,,
\end{align*}
where the last inequality follows from equation (6.51) in \cite{GiaquintaMartinazzi}. The rest of the proof is exactly the same.
\end{proof}

\subsection{Counterexample to continuity: proof of Proposition \ref{prop:weaknotcont}}
In this section we provide an example of a function $f\in W^{1,2}(B_{\frac12}, \A_2(\R))$ that is stationary with respect to the outer variations, but has a discontinuity point in $B_\frac14$.

We start with the following observation. Let $g \in C^1(\overline{\{g>0\}})\cap W^{1,2}(\Omega)$ non-negative such that $\Delta g=0$ on $\{g>0\}$ and $\partial \{ g>0 \}$ is $C^1$ regular, then $f=\a{g} + \a{-g} \in W^{1,2}(\Omega, \A_2(\R))$ is stationary with respect to outer variations. Indeed let $\psi(x,u)$ be any admissible variation then 
\begin{align*}
    \mathcal{O}(f,\psi) &= \int_{\{g>0\}} \sum_{i=0,1} \langle (-1)^iDg\colon D_x\psi(x,(-1)^ig)\rangle + \langle Dg \colon D_u\psi(x,(-1)^i g)Dg \rangle\\ 
    &= \int_{\{g>0\}} \sum_{i=0,1} (-1)^i \langle Dg \colon D (\psi(x,(-1)^ig))\rangle \\
    &= \int_{\partial \{g>0\}\cap \Omega} \langle \partial_\nu g \colon \psi(x,0)\rangle - \langle \partial_\nu g \colon \psi(x,0)\rangle = 0 \,.
\end{align*}
Notice also that the definition of $f$ implies that $\int_\Omega |Df|^2 = 2 \int_\Omega |Dg|^2$.

Next let $\{x_k\}_k\subset B_{\frac12}$, $r_k \downarrow 0$ satisfying 
\begin{itemize}
    \item[a)] $x_k \to x_0$, $|x_0|<\frac14$;
    \item[b)] $|x_k-x_l|>r_k+r_l$ for $k\neq l$.
\end{itemize}
For instance $x_k=\frac{1}{2k}$ with $r_k = 4^{-k}$ would work. 
For each $r_k$ choose $0<\rho_k<r_k$ such that 
\begin{itemize}
    \item[c)] $\left|\frac{\ln r_k}{\ln(\rho_k)}\right|<4^{-k}$;
    \item[d)] $ \int_{B_1\setminus B_{\rho_k}} \left|D\frac{\ln|x|}{\ln\rho_k}\right|^2 = \frac{2\pi}{|\ln\rho_k|}< 4^{-k}$.
\end{itemize}
Now we define by induction $g_0=1$ and \[g_{k}= \left(g_{k-1} + \frac{\ln|x-x_k|}{|\ln\rho_k|}\right)^+\,.\]
In particular $g_k\ge 0$ and $\Delta g_k=0$ on $\{g_k>0\}$.
Note that $0\le g_{k} \le g_{k-1} \le 1$ and $g_{k-1}-g_{k} < 4^{-k}$ on $B_{\frac12}\setminus B_{r_k}(x_k)$. Hence we have that 
\begin{align*} g_{k} \ge 1- \sum_{l=1}^{k} 4^{-l} \ge \frac23\qquad  \text{ on } B_{\frac12} \setminus \bigcup_{l=1}^{k} B_{r_l}(x_l)\\
\norm{D(g_k-g_{k-1})}_{L^2(B_1)} \le \norm{D\frac{\ln|x-x_k|}{\ln\rho_k}}_{L^2(B_1\setminus B_{\rho_k}(x_k))} \le 2^{-k}\,.\end{align*}
Since $\frac{\ln|x-x_k|}{|\ln\rho_k|}\le -1$ on $B_{\rho_k}(x_k)$, we deduce that $g_k=0$ on $B_{\rho_k}(x_k)$ and that $g_k$ satisfies the assumption of the observation above, hence 
\[ f_k = \a{g_k} + \a{-g_k}\] 
satisfies $\mathcal{O}(f_k,\cdot) =0$\,.
Due to the established estimates $f_k$ is a Cauchy sequence in $W^{1,2}(B_\frac12, \A_2(\R))$ and $C^0(B_\frac12\setminus \bigcup B_{r_k}(x_k), \A_2(\R))$ with limit $f$. In particular we deduce 
\begin{itemize}
    \item[i)] $\mathcal{O}(f,\cdot)=0$;
    \item[ii)] $|f|(x)\ge 2 \,\frac23$, for $|x-x_k|=r_k$, $k \in \N$
    \item[iii)] $|f|(x)=2\a{0}$ for $|x-x_k|<\rho_k$, $k\in \N$.
\end{itemize}
The last two points together imply that $x_0$ is a discontinuity point.
\qed

\section{Concentration compactness: proof of Theorem \ref{thm:comp}}

In this section we prove Theorem \ref{thm:comp}. We proceed by induction on $Q$.

\medskip

\noindent\emph{Base step:}
Up to subtracting a constant, this follows from the regularity of harmonic functions. 

\medskip

\noindent\emph{Inductive step:} Notice that the averages $\eta \circ f_k$ are harmonic, by outer variation equal to $0$, and moreover
\[
|Df_k|^2=|Df_k \ominus (\eta \circ f_k) |^2+Q\,|\eta \circ f_k|^2\,,
\]
so we can assume without loss of generality that $\eta\circ f_k=0$ for every $k$. We split the rest of the proof in two steps. 

\medskip

\noindent\emph{Step 1}:
We first prove the theorem under the additional assumption that for all $k$ we have
\begin{equation}\label{eq:proofunderL2bound}
    |f_k(0)|\leq C<\infty\,.
\end{equation}
By the equicontinuity, assumption (1), we know that there exists $f\in W^{1,2}(B_1,\Iq)$, such that, up to a not relabelled subsequence, $f_k$ converges to $f$ in $C^0$ and weakly in $W^{1,2}$. Next let 
\[
\mu=\lim_{k\to \infty}|Df_k|^2-|Df|^2\geq 0\,,
\]
the limit intended in the sense of measures, that is $\mu$ is the defect measure and it is non-negative by lower semicontinuity of the energy. We are going to prove that $\mu=0$, which gives the desired result.

Indeed, if $x\in \{x\in B_1\,:\,|f|>0 \}$, then by continuity of $f$ there exists a radius $r>0$, depending on $x$, such that 
\[
f(y)=\a{f^1(y)}+\a{f^2(y)}\qquad \forall y\in B_{r}(x)\,,
\]
with $f^i\in W^{1,2}(B_1,\I{Q_i})$, $Q_i\geq 1$. By the uniform convergence, for $k$ sufficiently large, we have
\[
f_k(y)=\a{f^1_k(y)}+\a{f^2_k(y)}\qquad \forall y\in B_{r}(x)\,.
\]
By Remark \ref{rem:natsplit}, we also have
\[
\G^2(f_k,f)=\G^2(f_k^1,f^1)+\G^2(f_k^2,f^2)\,,
\]
so that, $|f_k^i(x)|\leq C<\infty$. In particular $(f^i_k)_k$ satisfies the inductive assumptions, including \eqref{eq:proofunderL2bound}, and we have that 
\[
\lim_{k\to \infty}\int_{B_r(x)}|Df^i_k|^2=\int_{B_r(x)}|Df^i|^2\qquad i=1,2\,,
\]
so that $\mu(B_r(x))=0$.
This implies that
\begin{equation}\label{eq:mu1}
    \mu(\{x\in B_1\,:\,|f|>0 \})=0\,.
\end{equation}

To analyse $\mu$ on the set $\{ x \in B_1\colon |f|=0\}$ fix smooth non-decreasing $\theta$ vanishing on $(-\infty, 1]$ and equal to $1$ on $[2,\infty)$. Moreover let $\eta$ be a smooth function which is $1$ in the ball of radius $r$, less than or equal to $1$, and supported in $B_1$. Then we test the outer variation \eqref{eq:outer} with the test function
\[
\psi(x,u):=\eta(x)\,\theta\left(\frac{\ln|u|}{\ln\delta}\right)\, u\,.
\]
Since \[D_u\psi(x,u)=\eta(x)\,\left(\theta\left(\frac{\ln|u|}{\ln\delta}\right)\,\mathbf{1} + \frac{1}{\ln\delta}\,\theta'\left(\frac{\ln|u|}{\ln\delta}\right) \frac{u\otimes u}{|u|^2}\right)\]
we obtain from the outer variation
\begin{align*}
\int_{\{|f_k|<\delta^2\}}\, \eta\, |Df_k|^2
    &\leq \int \eta\, \theta\left(\frac{\ln|f_k|}{\ln\delta}\right)\, |Df_k|^2\\
    &=-\int  \theta\left(\frac{\ln|f_k|}{\ln\delta}\right)\, f_k\,Df_k\,\nabla \eta-\frac{1}{\ln\delta}\int \eta\, \theta'\left(\frac{\ln|f_k|}{\ln\delta}\right)\, \frac{(f_k\cdot Df_k)^2}{|f_k|^2}\\\ 
    &\leq \frac{1+\|\theta'\|}{|\ln\delta|}\,\left(\| Df_k\|_{L^2(B_1)}\|\nabla \eta\|_{L^2(B_1)}+\| Df_k\|^2_{L^2(B_1)}\right) \,.
\end{align*}
Passing to the limit we obtain
\[
\mu(B_r(x)\cap \{|f|<\delta^2\})\leq \int_{ \{|f|<\delta^2\}} \eta \,d(\mu+|Df|^2)\leq \frac{C}{|\ln\delta|}
\]
so that we can conclude
\[
\mu(\{ |f|=0\})=\lim_{\delta \to 0}\mu(\{ |f|<\delta^2\})=0\,,
\]
which combined with \eqref{eq:mu1} concludes the proof of Step 1.


\medskip
\noindent\emph{Step 2:} It remains to consider the case $\limsup_{k \to \infty} |f_k(0)| = \infty$. We apply Lemma \ref{lem:camilloemanuele} to $f_k(0)$ with parameter $\epsilon = \frac18$ and obtain points $S_k=\sum_{i=1}^{q_k} Q_k^i  \a{s_k^i}$. Passing to a subsequence, we may assume that $q_k, Q_k^i$ do not depend on $k$. 
By assumptions \[\frac{\beta\left(\frac18, Q\right)}{2Q}\; |f_k(0)| \le \beta\left(\frac18 ,Q\right) \;\diam(f_k(0)) \le \sep(S_k) \]
hence $\lim_{k\to \infty} \sep(S_k) = \infty$. By the equicontinuity and Remark \ref{rem:natsplit}, for sufficiently large $k$ the functions $f_k$ split, i.e. $f_k= \sum_{i=1}^q \a{f^i_k}$, with $f^i_k$ themselves satisfying the assumptions of Step 1, so nothing travels off to infinity. Therefore we can apply the inductive step once again and conclude.
\qed

\begin{corollary}[Compactness for $2$-d stationary maps]\label{cor:compactstationary}
	Let $f_k \in W^{1,2}(B_R, \Iqs)$, $B_R\subset \R^2$, be a sequence of stationary maps satisfying 
	\[ f_k(0)=Q\a{0} \quad\text{ and }\quad \sup_{k} \int_{B_R} |Df_k|^2 < \infty\,,\]
	then there exists a stationary map $f \in W^{1,2}(B_R, \Iqs)$ such that for every $0<r<R$
	\[\lim_{k \to \infty} \sup_{B_r} \G(f_k(x),f(x)) =0 \quad\text{ and } \quad\lim_{k \to \infty} \int_{B_R} ||Df_k|-|Df||^2 =0 \,.\]
\end{corollary}
\begin{proof}
	Since all the maps $f_k$ are stationary with uniformly bounded Dirichlet energy they are uniformly H\"older continuous, compare \ref{eq:holder}. Furthermore the assumption $f_k(0)=Q\a{0}$ ensures together with H\"older regularity ensures that the functions don't split. Hence we are in the setting of concentration compactness without splitting and the claimed convergence follows. But since the inner and outer variation are continuous with respect to strong $W^{1,2}$ convergence it follows that $f$ is stationary. 
\end{proof}

\section{Analysis of the singular set}

 In the first subsection we show how, given a stationary map, we can construct a new one whose points of cardinality $k$ are collapsed to $Q\a{0}$. In the second section we recall the frequency function and its consequences in our setting. Finally in the last subsection we conclude the proof of Theorem \ref{thm:dimension}.

\subsection{Splitting of a stationary map}\label{sec:splitting} Given a stationary map we construct a new one whose points of fixed cardinality are collapsed as follows:

\begin{proposition}\label{prop:decomposition}
	Let $f$ be a stationary map with ${\rm card}(f(0))=K$, that is
	\[f(0)= \sum_{k=1}^K Q_k \a{p_k}\,,\quad \text{with }\,p_k\neq p_l\,,\,\forall l\neq k\,. \]
	Then there exist $r_0>0$ and 
	\begin{enumerate}
		\item harmonic functions $h_k \colon B_{r_0} \to \R^n$, with $h_k(0)=p_k$,
		\item average free $Q_k$-valued maps $g_k \in W^{1,2}(B_{r_0}, \overset{\circ}{\mathcal{A}}_{Q_k}(\R^n))$ that are weakly stationary, i.e. stationary with respect to the outer variations, such that
	\end{enumerate}
	\begin{equation}\label{eq:decomposition1}f= \sum_{k=1}^K h_k \oplus g_k\quad \text{in }B_{r_0}\,.\end{equation}
	Furthermore, the new map 
	\begin{equation}\label{eq:decomposition2}
		\tilde{f} = \sum_{k=1}^K g_k\in W^{1,2}(B_{r_0},\Iqsn)
	\end{equation}
	has the properties that 
	\begin{enumerate}
	\setcounter{enumi}{3}
		\item $\tilde{f}$ is stationary with respect to inner and outer variations;
		\item $\tilde{f}(0)=Q\a{0}$;
		\item for all $z \in B_{r_0}$ with $\operatorname{card}(\supp f(z))=K$ we have $\tilde{f}(z)=Q\a{0}$;
    	\item $\sing(f)\cap B_{r_0} \subset \sing(\tilde{f})\cap B_{r}$;
	\end{enumerate}
\end{proposition}

\begin{proof}
	Due to the continuity of $f$ we may choose $r_0>0$ such that
	\[\G(f(x),f(0))<\frac{1}{16} \sep(f(0)) \quad \forall |z|<r_0\,. \]
	This implies that $f$ splits, see for instance Remark \ref{rem:natsplit}, i.e. there are are $Q_k$-valued continuous functions $f_k$ on $B_{r_0}$ such that 
	\[f(z) = \sum_{k=1}^K f_k(z)\,,\qquad z\in B_{r_0}\,. \]
	Since $f$ is weakly stationary and $\dist(\spt(f_k(z)),\spt(f_l(z)))> \frac{1}{16} \sep(f(0))$ on $B_{r_0}$, we deduce that each $f_k$ is itself weakly stationary, i.e.
	\[\mathcal{O}(f_k, \psi)=0 \qquad \forall k\in \N\,. \]
	In particular this implies that the averages $h_k=\eta\circ f_k $ are all harmonic in $B_{r_0}$, i.e. $\Delta h_k =0$. 
	
	Now we set $g_k= f_k \ominus h_k$ and we observe that, since $h_k$ is harmonic, by Lemma \ref{lem.OuterInvariance} and \ref{lem.InnerInvariance}, each $g_k$ satisfies
	\begin{align*}
	\mathcal{O}(g_k, \psi)&= \mathcal{O}(f_k, \psi_{h_k}) =0 \qquad\text{ for all outer variations } \psi \\
	\mathcal{I}(g_k, \phi) &= \mathcal{I}(f_k,\phi) \qquad\text{ for all domain variations } \phi\,.
		\end{align*}
		In particular we have established (1), (2) and \eqref{eq:decomposition1}.
		
		Next we notice that $\tilde{f}$ is stationary since
		\[ 0=\mathcal{I}(f,\phi)=\sum_{k=1}^K \mathcal{I}(f_k, \phi) = \sum_{k=1}^K \mathcal{I}(g_k, \phi) = \mathcal{I}(\tilde{f},\phi)\, .\]
		and analogously, given any admissible outer variation $\psi \in C^\infty(\Omega_x\times \R^{n}_u; \R^{n})$, we have
		\[ \mathcal{O}(\tilde{f}, \psi) = \sum_{k=1}^K \mathcal{O}(g_k, \psi)=0\,.\]
	Hence (4) is proven. (5) follows by construction.
	To deduce (6) we only observe that due to \eqref{eq:decomposition1} and the first two displayed equations in the proof, we have that $\operatorname{card}(\spt(f(z))\ge K$ for all $z\in B_{r_0}$, and equality implies that $g_k(z)=Q_k\a{0}$ for all $k$, since the $g_k$ are average free.
	Arguing by contraposition, (7) is an immediate consequence of the combination of Lemma \ref{lem:constantcardinalety} and Corollary \ref{cor:sing=card non continuous} below.
\end{proof}

\begin{lemma}\label{lem:constantcardinalety}
	Let $f\colon \Omega \to \Iqs$ be continuous on a simply connected domain $\Omega\subset \R^m$  then the following are equivalent:
	\begin{enumerate}
		\item $x \mapsto \operatorname{card}(\spt(f(x))$ is constant. 
		\item there are maps $g_j \in C^0(\Omega, \R^n)$, $j=1, \dotsc, Q$, such that $f=\sum_{j=1}^Q \a{g_j}$ and either $g_j\equiv g_i$ or $g_j(x) \neq g_i(x)$, for every $x \in \Omega$. 
	\end{enumerate}
\end{lemma}
\begin{proof}
	Clearly (2) implies (1). To show that (1) implies (2), we  assume that $\operatorname{card}(\spt(f(x)) \equiv  K$ on $\Omega$.  
	
	First we claim that each $x\in \Omega$ has a neighbourhood $U_x\subset \Omega$ where $f$ admits a decomposition as described in (2). \\
	This is true, since $f(x)= \sum_{j=1}^K Q_j \a{t_j}$, with $t_i\neq t_j$ for $i\neq j$, and so by the continuity of $f$ we can apply Remark \ref{rem:natsplit} to find in a neighbourhood $U_x$ where $f$ decomposes into $K$ maps, i.e. there are continuous maps $\tilde{g}_j\colon U_x \mapsto \mathcal{A}_{Q_j}(\R^n)$ such that $f=\sum_{j=1}^K \tilde{g}_j$ and $\spt(\tilde{g}_i) \cap \spt(\tilde{g}_j)=\emptyset$ for $i\neq j$. This implies that $\operatorname{card}(\spt(f(y)) \ge K$ on $U_x$. Since by assumption  $\operatorname{card}(\spt(f(y))=K$, we deduce that $\tilde{g}_j=Q_j \a{g_j}$ on $U_x$, thus proving the claim.
	
	Hence we can consider the set $\spt(\mathbf{G}_f(\Omega))$, that is the support of the graph of $f$, as a covering space of the domain $\Omega$. Since $\Omega$ is simply connected, we obtain the desired  decomposition.  
\end{proof}

\begin{corollary}\label{cor:sing=card non continuous}
	Let $f\colon \Omega \to \Iqs$ be continuous and weakly stationary. Then the following are equivalent \begin{enumerate}
		\item $y \mapsto \operatorname{card}(\spt(f(y))$ is constant in a neighbourhood of $x$,
		\item $x\in \reg(f)$.
	\end{enumerate}
\end{corollary}
\begin{proof}
	That (2) implies (1) follows directly from the implication (2) implies (1) in Lemma \ref{lem:constantcardinalety}. 
	
	To prove the converse we may assume that $U_x\subset \Omega$ is a simply connected domain where $y \mapsto \operatorname{card}(\spt(f(y))$ is constant. Hence we find a decomposition as described in (2) of Lemma \ref{lem:constantcardinalety}. But then we can use the outer variation to deduce that each of these maps must be harmonic.  
\end{proof}

\subsection{The frequency function and blow-ups}
Let $f$ be a $Q$-valued function, and assume that $f(0)=Q\a{0}$ and $\int_{B_r}|Df|^2>0$. Following \cite[section 3.4]{DS}, we define the quantities 
\begin{equation}\label{eq:frequencyfunction}
	D_{x,f}(r)= \int_{B_r(x)} |Df|^2, \quad H_{x,f}(r) = \int_{\partial B_r(x)} |f|^2\quad \text{and} \quad I_{x,f}(r) = \frac{rD_{x,f}(r)}{H_{x,f}(r)}\,.
\end{equation}
Since the proof concerning the monotonicity of the frequency function only relies on the outer and inner variation formulas, \eqref{eq:outer}-\eqref{eq:inner}, we have: 

\begin{theorem}\cite[Theorem 3.15]{DS}\label{thm:frequencymon}\
Let $f$ be stationary and $x\in \Omega$. Either there exists $\rho$ such that $f|_{B_\rho(x)}=Q\a{0}$ or $I_{x,f}(r)$ is an absolutely continuous non-decreasing function on $]0, \dist(x,\partial\Omega)[$.  
\end{theorem}

Next we define the blow-ups of $f$ at $y$ by $\displaystyle{
f_{y,\rho}(x):=\frac{\rho^{\frac{m-2}2}\,f(\rho\,x+y)}{\sqrt{D_{y,f}(\rho)}}}$. Then we have the following

\begin{theorem}[{\cite[Theorem 3.19]{DS}}]\label{thm:bu}
Let $f \in W^{1,2}(B_1,\Iq)$ be a stationary map, with $B_1\subset \R^2$. Assume $f(0)= Q\a{0}$ and $D_{f}(\rho)>0$ for every $\rho\leq 1$. Then, for any sequence $\{f_{\rho_k}\}$ with $\rho_k \downarrow 0$, a subsequence, not relabeled, converges locally uniformly to a function $g : \R^2 \to \Iq$ with the following properties:
\begin{enumerate}
    \item $D_{0,g}(1)=1$ and $g$ is stationary;
    \item $\displaystyle{g(x)= |x|^{\alpha}\,g\left(\frac{x}{|x|}\right)}$, where $\alpha=I_{0,f}(0)>0$ is the frequency of $f$ at $0$. 
\end{enumerate}

\end{theorem}

\begin{proof}
The proof follows from the same arguments as in the proof of \cite[Theorem 3.19]{DS}, replacing Theorem 3.9 about H\"older continuity of minimizers with Corollary \ref{cor:hoelder} on the H\"older continuity of stationary maps, and Proposition 3.20 on the compactness of minimizers, with our Corollary \ref{cor:compactstationary} for stationary maps.
\end{proof}

Finally we will need the following elementary lemma on the classification of homogeneous stationary map on $\R$.

\begin{lemma}\label{lem:classificationin1d}
	Let $h \in W^{1,2}(\R, \Iqs)$ be an $\alpha$-homogeneous map that is stationary with respect to the outer variation then 
	\begin{enumerate}
		\item $\alpha=1$ 
		\item there are two points $T_\pm \in \Iqs$ such that 
		\[ h(t) = \begin{cases}
			t T_+ &\text{ if } t>0\\
			t T_- &\text{ if } t\le 0
		\end{cases}\,.\]
	\end{enumerate}
	Furthermore $h$ is stationary with respect to inner variations if and only if $|T_+|=|T_-|$.
\end{lemma}

In general the problem of classifying homogeneous stationary maps seems rather difficult even in dimension $2$, as illustrated by Remark \ref{rem:examplehomogeneous}.

\subsection{Estimate on the size of the singular set: proof of Theorem \ref{thm:dimension}}
Having established all the needed tools, we can combine them to show the estimate on the size of the singular set of a stationary map $f$, i.e. $\dim(\sing(f))\le 1$. In fact with the established tools the argument is nowadays almost ``classical'' and is very close to the original argument presented in \cite[section 3.6]{DS}. 
Hence we only outline the argument and highlight the needed adaptations.

Having Corollary \ref{cor:sing=card non continuous} in mind, it is natural to decompose the singular set into the subsets 
\[ S_k=\sing(f) \cap \{\operatorname{card}(f)=k\}\,,\qquad  \text{ for } k=1, \dotsc, Q\,.\]
Note that, since $x \mapsto \operatorname{card}(f(x))$ is lower semicontinuous, all the sets $S_k$ are relatively open in $\sing(f)$. Furthermore $S_Q=\emptyset$ and $S_1$ corresponds to the set $\Sigma_Q$ studied in \cite[section 3.6]{DS}.

Assume now by contradiction that $\mathcal{H}^t(\sing(f))>0$, for some $t>1$. This implies that there is at least one $k<Q$, with $\mathcal{H}^t(S_k)>0$. Hence there is a point $x_0 \in \Omega$ with positive density, i.e. 
$$\limsup_{r \to 0} \frac{ \mathcal{H}^t(S_k\cap B_r(x_0))}{r^t} >0\,.
$$
After translation we may assume that $x_0=0$ and note that, due to Corollary \ref{cor:sing=card non continuous}, $x_0$ cannot be an interior point of $S_k$. We apply Proposition \ref{prop:decomposition} to $f$ around $0$ and obtain the map $\tilde{f}$, such that $\tilde{f} \neq Q \a{0}$, since $0$ is not an interior point of $S_k$, and $\eta \circ \tilde{f} =0$ by (2). Furthermore due to (6) and (7) we have 
\begin{equation}\label{eq:positivedensity} \limsup_{r\to 0} \frac{ \mathcal{H}^t(\sing(\tilde{f})\cap \{\tilde{f}=Q\a{0}\}\} \cap B_r)}{r^t} >0 \,.  \end{equation}
Now the conclusion follows by the same arguments presented in \cite[Subsection 3.6.2]{DS}: let $r_k \downarrow 0$ be a subsequence realising the $\limsup$ in \eqref{eq:positivedensity} and consider the corresponding blow-up sequence $\tilde{f}_{r_k}$. By Theorem \ref{thm:bu}, we find a nontrivial $\alpha$-homogeneous stationary tangent map $g \colon \R^2 \to \Iqsn$.  Moreover  by \eqref{eq:positivedensity}, $g$ satisfies 
$$\mathcal{H}^t_\infty( B_1 \cap \{ g=Q\a{0}\})>0.$$ 
Therefore there exists $y \in \partial B_1 \cap \partial \{g=Q\a{0}\}$ once again with positive $\mathcal{H}^t_\infty$-density. Since this must be a singular point we can perform a second blow-up as described in \cite[Lemma 3.24]{DS}. However the corresponding tangent function $k\in W^{1,2}(\R^2, \Iq(\R^n))$ is nontrivial, homogeneous, depending only on one variable and such that $\mathcal{H}^t_\infty( B_1 \cap \{ k=Q\a{0}\})>0$, contradicting Lemma \ref{lem:classificationin1d}. This proves the estimate on the singular set.

\begin{remark}\label{rm:mink}
It is very likely that applying the arguments presented in \cite{DMS} to a $2$-dimensional stationary map $f$ would actually lead to Minkowski-type bounds 
\[ 
\mathcal{L}^2\left(B_{r_0} \cap B_{r}(\{f=Q\a{0}\})\right) \le C(\tilde{f},\rho_0) r\,. \]
Additionally we expect that $\{f=Q\a{0}\}$ is countable $1$-rectifiable in $B_{r_0}$. In fact, applying the reasoning to $\tilde{f}$ should lead to the rectifiability of $S_k\cap B_{r_0}$ and therefore to the rectifiability of the singular set of any $2$-dimensional stationary map $f$.
\end{remark}

\bibliographystyle{plain}
\bibliography{continuous2d.bib}

\end{document}